\documentclass[onefignum,onetabnum]{siamart190516}

\usepackage{url}
\usepackage{amsfonts}
\usepackage[T1]{fontenc}

\newcommand{\citeasnoun}[1]{Ref.~\cite{#1}}
\renewcommand{\eqref}[1]{\cref{eq:#1}}

\newcommand{\figref}[1]{\cref{fig:#1}}

\newcommand{\secref}[1]{\cref{sec:#1}}

\begin{document}

\title{Modified discrete Laguerre polynomials for efficient computation of exponentially bounded Matsubara sums} \author{
Guanpeng~A.~Xu and Steven~G.~Johnson\thanks{Department of Mathematics, Massachusetts Institute of Technology, Cambridge, MA (\email{stevenj@mit.edu})}
}

\maketitle

\begin{abstract}
    We develop a new type of orthogonal polynomial, the modified discrete Laguerre (MDL) polynomials, designed to accelerate the computation of bosonic Matsubara sums in statistical physics. The MDL polynomials lead to a rapidly convergent Gaussian ``quadrature'' scheme for Matsubara sums, and more generally for any sum $F(0)/2 + F(h) + F(2h) + \cdots$ of exponentially decaying summands $F(nh) = f(nh)e^{-nhs}$ where $hs>0$.  We demonstrate this technique for computation of finite-temperature Casimir forces arising from quantum field theory, where evaluation of the summand $F$ requires expensive electromagnetic simulations.  A key advantage of our scheme, compared to previous methods, is that the convergence rate is nearly independent of the spacing $h$ (proportional to the thermodynamic temperature).  We also prove convergence for any polynomially decaying $F$.

\end{abstract}



\section{Introduction}
\label{sec:introduction}
Our work develops an efficient $N$-point scheme, analogous to Gauss--Laguerre quadrature, to rapidly evaluate exponentially decaying sums of the form
\begin{equation}
\sum\limits_{n=0}^\infty {}' \underbrace{f(nh) e^{-nhs}}_{F(nh)} h \approx \sum_{k=1}^N w_k F(x_k) \, ,
\label{eq:sum}
\end{equation}
where $\sum'$ denotes that the $n=0$ term is weighted by $1/2$, exploiting knowledge of an asymptotic exponential decay rate~$s > 0$ of the summand $F(x)$ to derive quadrature points $x_k$ and weights $w_k$ for any desired order $N$ of accuracy. This calculation arises for bosonic Matsubara sums in statistical physics~\cite{Lifshitz80}---for example, in the computation of finite-temperature Casimir forces from quantum field theory~\cite{Lifshitz80,Rodriguez2011,Johnson11}---in which case $h$ is proportional to the temperature $T$, evaluation of $F$ can involve expensive partial-differential equation (PDE) solutions, and the asymptotic decay rate~$s$ can be determined from a Born approximation~\cite{maghrebi2011diagrammatic}.  Even though \eqref{sum} is exponentially decaying, direct evaluation of this sum requires many summand evaluations in the common case where $hs$ is small, corresponding to low temperatures.  In the $h\to 0^+$ limit, \eqref{sum} becomes $\int_0^\infty f(x) e^{-sx} dx = \int_0^\infty f(y/s) e^{-y} dy / s$ and can be evaluated efficiently by Gauss--Laguerre quadrature~\cite{KytheSchaferkotter04,Johnson11}, so our goal was to develop a ``Gaussian summation'' scheme (equivalent to Gaussian quadrature with a discrete measure~\cite{engblom2006gaussian}) that extends Gauss--Laguerre quadrature to $h > 0$ for this type of sum (\secref{derivation}).  We demonstrate our ``modified discrete Laguerre'' (MDL) approach on a typical Casimir-force computation (\secref{casimir}) shown in \figref{graphIntegrand}, where evaluating $F$ requires the solution of Maxwell's equations~\cite{Johnson11, reid2013fluctuating, SCUFF2}.
We found that MDL summation requires many fewer points $N$ than naive summation at low temperatures, and also outperforms a competing Gaussian-quadrature scheme from the literature on Matsubara sums~\cite{monien2010gaussian, karrasch2010finite} that did not exploit knowledge of the asymptotic exponential decay rate. Ours was the only technique whose performance did not degrade as temperature ($\sim h$) decreased.  We anticipate that similar benefits from  MDL summation should apply to other physical problems involving Matsubara sums, e.g. for density functional theory where various Matsubara-summation methods have been proposed~\cite{Kaltak2020,karrasch2010finite,Kananenka2016}, and more generally for evaluating any slowly converging exponentially weighted sums.

\begin{figure}[tb]
  \centerline{\includegraphics[width=1.3\columnwidth]{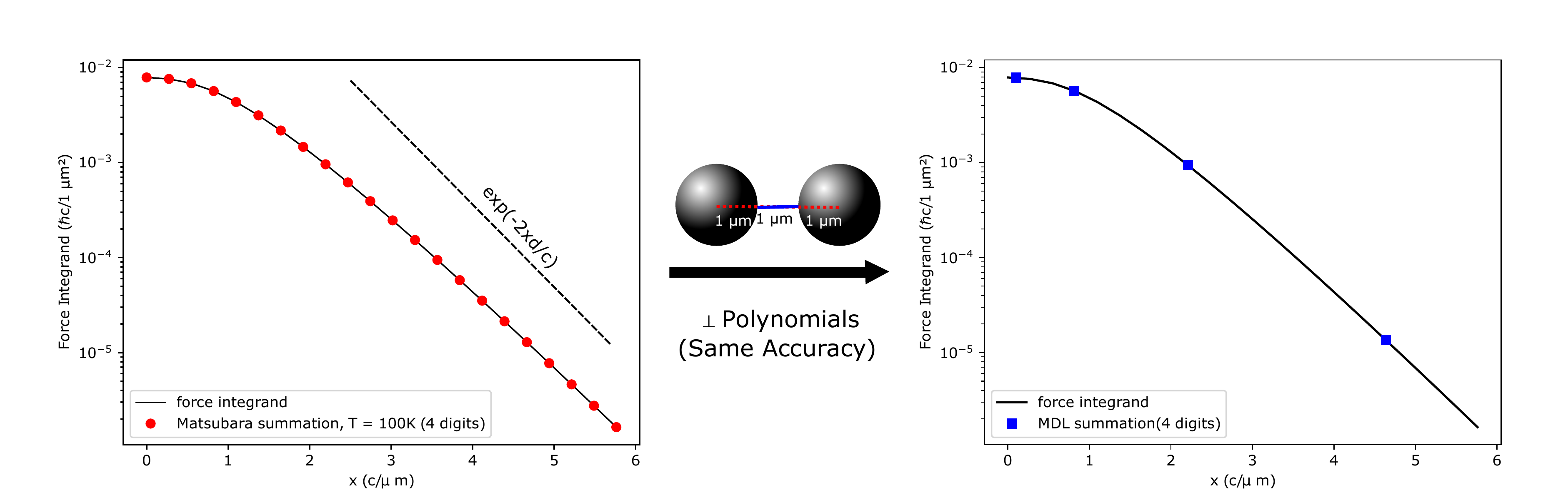}}
  \caption{We aim to reduce the number of summand evaluations (left) through a quadrature rule (right). The 4-point quadrature rule attains identical accuracy to a 20-term partial Matsubara sum at $T = 100$~K.   We exploit the fact that the asymptotic decay rate (black dashed line) is known to be $e^{-2xd/c}$ for objects with minimum separation $d$ ($= 1\,\mu m$ here).}
  \label{fig:graphIntegrand}
\end{figure}

The key to an efficient summation scheme for \eqref{sum}, as for all Gaussian-quadrature methods~\cite{KytheSchaferkotter04}, is the construction of a set of orthogonal polynomials---in this case for the discrete inner product
\begin{equation}
\langle f, g \rangle_{h,s}'= \sum\limits_{n=0}^\infty {}' f(nh)g(nh) e^{-nhs} h \, .
\label{eq:innerproduct1}
\end{equation}
We call our resulting polynomials the ``modified discrete Laguerre'' (MDL) polynomials: they reduce to the ordinary Laguerre polynomials~\cite{KytheSchaferkotter04} in the limit $h\to 0^+$ for $s=1$, and they are closely related to the discrete Laguerre (DL) polynomials derived in~\citeasnoun{gottlieb1938concerning}, which differ only in the $n=0$ term of the inner product (using $\sum$ instead of $\sum'$).  Regardless of whether the inner product is an integral or a discrete sum, orthogonal polynomials (described by their three-term recurrence, derived in \secref{derivation}) immediately yield the weights $w_n$ and points $x_n$ in \eqref{sum} (e.g. by the Golub--Welsch algorithm~\cite{golub1969calculation, trefethen1997numerical,Townsend15}) to exactly sum/integrate polynomials up to degree $2N-1$ and exhibit rapid convergence for a broad class of  smooth functions $f$~\cite{bultheel2000convergence,Trefethen2008}.  In principle, one could use the original DL polynomials to evaluate \eqref{sum} by separately subtracting $F(0)/2$, but this would require one additional function evaluation beyond the $x_k >0$ quadrature points. Worse, for the Casimir case, $F(0)$ involves a singularity that must be evaluated by an expensive $x \to 0^+$ numerical-extrapolation procedure~\cite{Johnson11}, whereas MDL summation only evaluates $F$ only at $x > 0$ and implicitly computes the contribution of the limit.  In contrast to direct summation and previous Gaussian quadrature techniques, our MDL scheme maintains good convergence even at low temperature (small $h$) where the direct summand decays more and more slowly.
Finally, in \cref{sec:convergenceproof}, we adapt previous work on the convergence of Gaussian quadrature~\cite{uspensky1928convergence} for improper integrals in order to prove convergence of MDL quadrature for any $f(x)$ decaying faster than $x^{-1 - \epsilon}$ for $\epsilon > 0$. For fermionic Matsubara summation ($\sum$ instead of $\sum'$ of $F(nh+h/2)$), one would simply apply an analogous quadrature scheme based on the original DL polynomials instead of the MDL polynomials.

\section{MDL polynomials and their three-term recurrence} \label{sec:derivation}

In this section, we derive certain key properties of the MDL polynomials, culminating in the three-term recurrence of \cref{threeterm}.  Given the three-term recurrence coefficients, standard techniques such as the Golub--Welsch algorithm~\cite{golub1969calculation, atkinson2008introduction, Townsend15, trefethen1997numerical} immediately yield the quadrature points and weights of \secref{properties}, shown in \figref{graphIntegrand} and employed in \secref{casimir}.  We begin with some definitions and preliminary results adapted from~\citeasnoun{gottlieb1938concerning}, and outline the steps leading to \cref{threeterm}.  The details of the inductive proofs are given in \cref{sec:proofs}.

For physical applications or for comparison to Gauss--Laguerre quadrature, it is convenient to separate the decay rate $s$ from the spacing $h$ in \eqref{sum}, but in deriving the MDL polynomials it is more natural to combine these into a single parameter $\tau = e^{hs} > 1$ and omit the overall $h$ scale factor; we will recover the original sum by a change of variables in \secref{generalization}.
In terms of $\tau$, we define the inner products $\langle \cdot, \cdot \rangle$ and $\langle  \cdot, \cdot  \rangle'$ of two real-valued functions $f$ and $g$ as follows, analogous to \eqref{sum}:
\begin{equation}
    \langle f, g \rangle (') = \sum\limits_{n=0}^\infty (') \tau^{-n}f(n) g(n) \, ,
\end{equation}
where the notation $\sum\limits_{i=0}^L {}^\prime s_i$ represents the summation $\frac{s_0}{2}+\sum\limits_{i=1}^L s_i$. The discrete Laguerre (DL) polynomials of~\citeasnoun{gottlieb1938concerning} are orthogonal with respect to the $\langle \cdot, \cdot \rangle$. Here, we introduce ``modified'' discrete Laguerre (MDL) polynomials that are orthogonal with respect to $\langle \cdot, \cdot \rangle '$. We denote the DL and MDL polynomials as $L_n(x)$ and $L_n'(x)$, respectively.

Below, we will make use of the following results proved in~\citeasnoun{gottlieb1938concerning}. The DL polynomials for $n\ge 0$ may be written in the closed form
\begin{equation}
L_k(n) = \sum\limits_{i=0}^k (-1)^{i}\binom{k}{i}\frac{(1-\frac{1}{\tau})^i}{i!}\prod\limits_{j=1}^i(n+j).
\label{eq:DL}
\end{equation}
Furthermore, the DL polynomials satisfy the following normalization condition:
\begin{equation}
L_n(0) = \frac{1}{\tau^n},  \; \; \; \;  \langle L_n, L_n\rangle= \frac{1}{\tau^{n-1}(\tau-1)}.
\label{eq:normalize}
\end{equation}

We now derive the coefficients of the three-term recurrence for the MDL polynomials $L_n'$. We proceed by analyzing the difference between $L_n'$ and $L_n$ using induction on $n$, arriving at the following results:

\begin{lemma}
For all nonnegative integers $n$, we have: $$ L_n' = L_n + \frac{1}{2}\frac{1}{\tau^n}\sum\limits_{i=0}^{n-1} \frac{L_i'(0)L_i'}{\langle L_i', L_i'\rangle'}. $$
\label{GS}
\end{lemma}
\begin{proof}
We apply the Gram--Schmidt procedure~\cite{trefethen1997numerical} to re-orthogonalize $L_n$ with respect to the modified inner product $\langle \cdot, \cdot \rangle'$, using \eqref{normalize}:
\begin{align*}
    L_n' &= L_n - \sum\limits_{i=0}^{n-1} \frac{\langle L_n, L_i'\rangle'}{\langle L_i', L_i'\rangle'}L_i' = L_n +\frac{1}{2} \sum\limits_{i=0}^{n-1} \frac{L_n(0)L_i'(0)}{\langle L_i', L_i'\rangle'}L_i' \\&= L_n +\frac{1}{2} \sum\limits_{i=0}^{n-1} \frac{L_i'(0)L_i'}{\langle L_i', L_i'\rangle'\tau^n} = L_n + \frac{1}{2}\frac{1}{\tau^n}\sum\limits_{i=0}^{n-1} \frac{L_i'(0)L_i'}{\langle L_i', L_i'\rangle'}.
\end{align*}
\end{proof}

The following statement is proved by induction on $n$ in \cref{sec:proofs}:
\begin{lemma} \label{MDLscale}
For all nonnegative integers $n$, we have: $$L_n'(0) = \frac{2}{1+\tau^n},  \; \; \; \;  \langle L_n', L_n'\rangle ' = \frac{1+\tau^{n+1}}{\tau^n(1+\tau^{n})(\tau-1)}. $$
\label{lemma:MDLscale}
\end{lemma}

We now derive the three-term recurrence for the $L_n'$ polynomials.
\begin{theorem}\label{recurrence}
For all nonnegative $n$, the MDL polynomials $L_n'$ satisfy the recurrence $$ L_{n+1}'(x) = (-\alpha_nx + \beta_n)L_n'(x) - \gamma_nL_{n-1}'(x),$$ where
\begin{align*} &\alpha_n = \frac{\tau-1}{(n+1)\tau} \\&\beta_n = \frac{1+\tau^n}{1+\tau^{n+1}} + \frac{n}{\tau(n+1)}\frac{1+\tau^{n+1}}{1+\tau^n} \\& \gamma_n = \frac{n}{n+1}\frac{(1+\tau^{n+1})(1+\tau^{n-1})}{\tau(1+\tau^n)^2}.\end{align*}
\end{theorem}
\begin{proof} See \cref{sec:proofs}.
\end{proof}

In order to apply the MDL polynomials to Gaussian-quadrature applications, it is necessary to define rescaled orthonormal variants $\hat{L}_n' = L_n'/\sqrt{\langle L_n', L_n' \rangle}$.  Using Lemma~\ref{lemma:MDLscale}, we derive the rescaling $$\hat{L}_n' = L_n' \cdot (-1)^n \sqrt{\frac{\tau^n(1+\tau^n)(\tau-1)}{1+\tau^{n+1}}} .$$

By straightforward algebraic substitution of this rescaling into \cref{recurrence}, we obtain a recurrence for the normalized MDL polynomials (whose coefficients appear directly in the Jacobi matrix used for Gaussian quadrature in \secref{generalization}):

\begin{theorem}
\label{threeterm}
Define two sequences $\{\hat{\alpha}_n\}_n, \{B_n\}_n$ as follows:
\begin{align*}
    \hat{\alpha}_n &= \frac{(n+1)\tau}{\tau-1}\left(\frac{1+\tau^n}{1+\tau^{n+1}} + \frac{n}{\tau(n+1)}\frac{1+\tau^{n+1}}{1+\tau^n}\right) \\
    \hat{\beta}_n \, , &=\frac{(n+1)\tau}{\tau-1}\sqrt{\frac{(1+\tau^n)(1+\tau^{n+2})}{\tau(1+\tau^{n+1})^2}} \, .
\end{align*}

Then the normalized polynomials $\hat{L}_n'$ satisfy the recurrence \begin{equation*}
    x\hat{L}_n'(x) = \hat{\beta}_n\hat{L}_{n+1}'(x) + \hat{\alpha}_n\hat{L}_n'(x) + \hat{\beta}_{n-1}\widehat{L_{n-1}'}(x).
\end{equation*}
\end{theorem}
We may see with direct calculation of coefficients that $\hat{\alpha}_n$ and $\hat{\beta}_n$ are positive, monotonically increasing, and asymptotically proportional to $n$ for any fixed $\tau > 0$.
\section{MDL quadrature for Matsubara sums}
\label{sec:generalization}
In order to evaluate the Matsubara sum \eqref{sum}, we now consider a broader class of sums, characterized by the generalized (M)DL inner product \begin{equation} \langle f,g \rangle_{h,s}(')= \sum\limits_{n=0}^{\infty}(') \, f(nh)g(nh)e^{-nhs}h.  \label{eq:generalizedform}
\end{equation}
Here, we may set $\tau = e^{hs}$; in this case, the generalized (M)DL polynomials correspond to a simple rescaling with respect to the original (M)DL polynomials. In particular, we define generalized MDL polynomials for this inner product via a change of variables: $$\hat{L}_k^{(h,s)\prime}(nh) = \hat{L}_k'(n).$$ By substituting this change of variables into \cref{threeterm}, we immediately see that $\hat{L}_k^{(h,s)\prime}(nh)$ satisfies an analogous three-term recurrence with coefficients $h \hat{\alpha}_n$ and $h \hat{\beta}_n$.

We compute the quadrature points (roots of $\hat{L}_k^{(h,s)\prime}$) and weights from the eigenvalues and eigenvectors of the symmetric tridiagonal Jacobi matrix (the Golub--Welsch algorithm):\begin{equation}J_N= h \begin{pmatrix}
     \hat{\alpha}_0 &      \hat{\beta}_0 &      &     &   & \\
     \hat{\beta}_0 &    \hat{\alpha}_1 &      \hat{\beta}_1 &        &   &  \\
        &    \hat{\beta}_1 &    \ddots &     \ddots &        &  \\
        &  & \ddots &  \ddots &  \hat{\beta}_{N-3}&        \\
  &  &       & \hat{\beta}_{N-3} & \hat{\alpha}_{N-2} &  \hat{\beta}_{N-2} \\
 &  &  &        & \hat{\beta}_{N-2} & \hat{\alpha}_{N-1}
\end{pmatrix} \, \label{eq:jacobi}.
\end{equation}
In the following subsections, we examine the properties of these roots and the resulting quadrature scheme for computing the sum~\cref{eq:jacobi}.

\subsection{Roots and asymptotic properties of MDL polynomials} \label{sec:properties}
The roots of $\hat{L}_N^{(h,s)\prime}$ correspond to the eigenvalues of the $N\times N$ Jacobi matrix~\eqref{jacobi}~\cite{golub1969calculation,trefethen1997numerical}.  These roots are plotted as a function of $hs = \ln(\tau)$ in \figref{MDLroots}, and exhibit asymptotic behaviors that we briefly explain here.

\begin{figure}[t]
  \centering
  \begin{minipage}[t]{0.48\textwidth}
    \includegraphics[width=\textwidth]{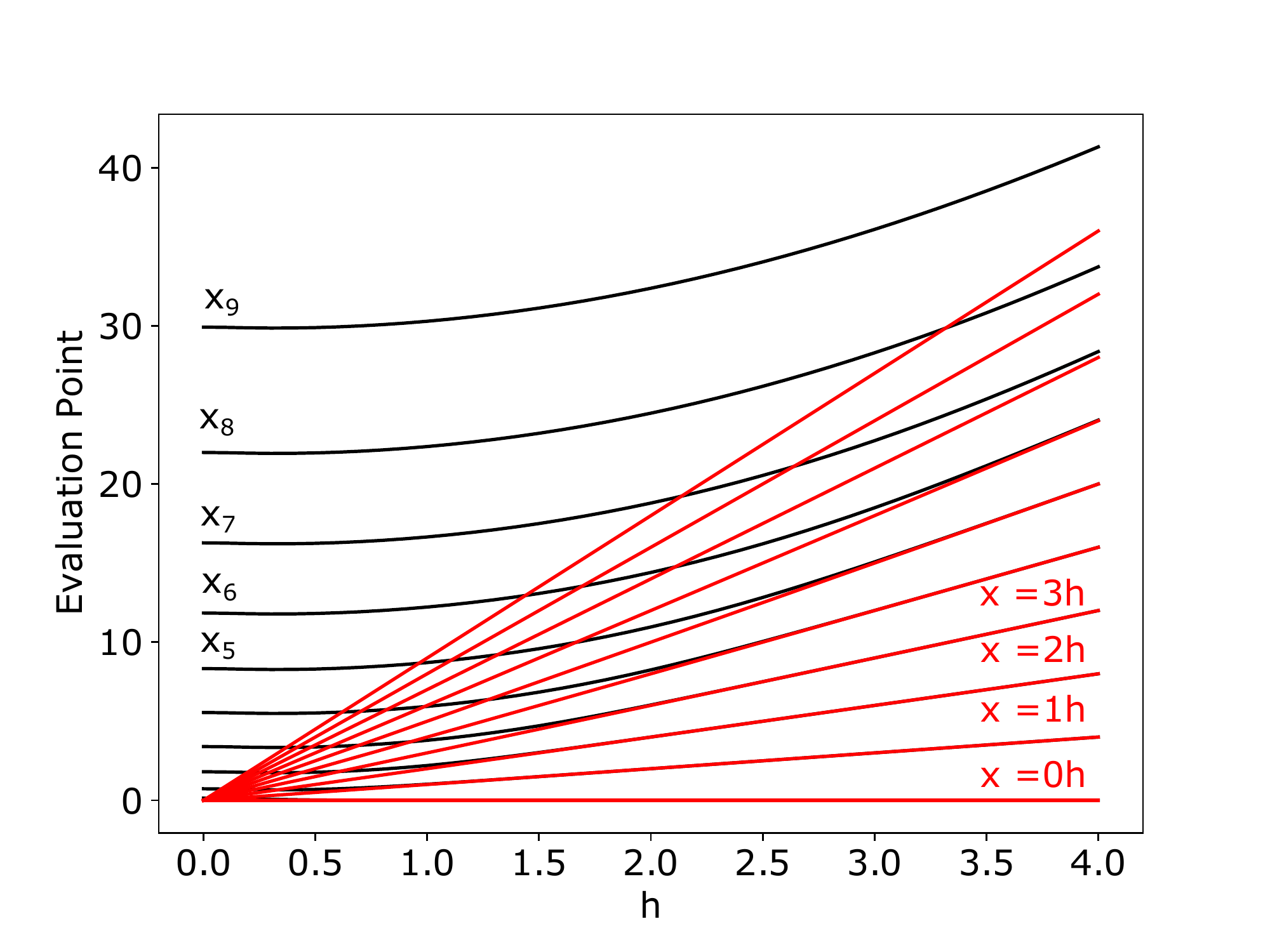}
    \caption{The roots of the modified discrete Laguerre (MDL) polynomials for $n = 10$, $s=1$, and $\ln (\tau) = hs \in [0, 4]$.  \label{fig:MDLroots}}
  \end{minipage}
  \hfill
  \begin{minipage}[t]{0.48\textwidth}
    \includegraphics[width=\textwidth]{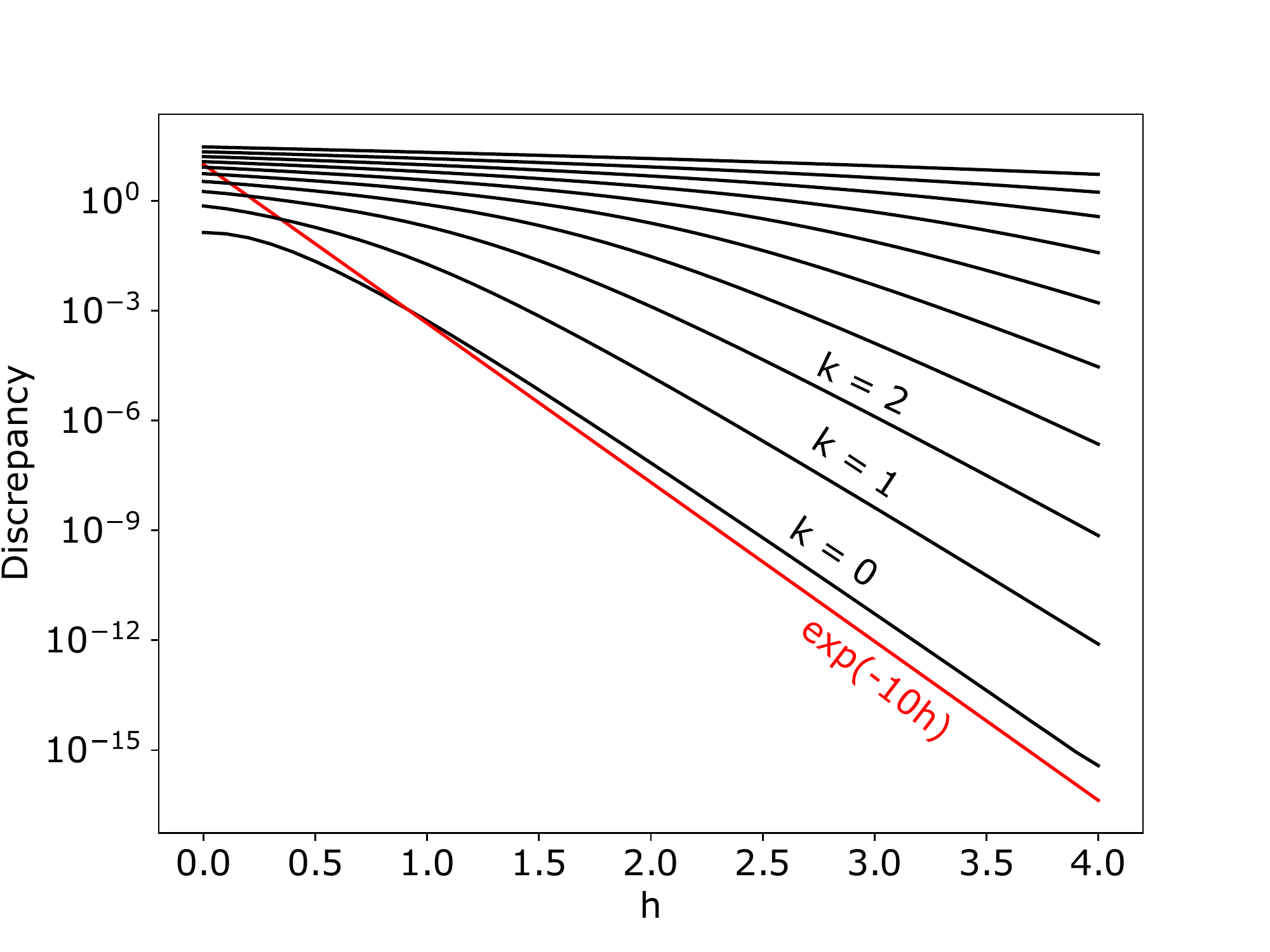}
    \caption{Differences $x_k - kh$ between the roots of the MDL polynomials $x_k$ and the Matsubara points $kh$ for the same $n,h,s$. \label{fig:MDLrootsdiff}}
  \end{minipage}
\end{figure}

Due to Proposition~2.3 of \citeasnoun{engblom2006gaussian}, the roots $x_k$ satisfy $kh < x_k < (k+1)h$ (for $k=0,1,\ldots$), and in the limit of large $h$ (high temperatures) the roots $x_k$ approach $kh$ from above. Correspondingly, the quadrature scheme for~\eqref{sum} reduces to naive summation in this limit.  (Intuitively, in the large-$h$ regime the summand weight $e^{-shn}$ decays so rapidly that naive summation is nearly optimal.)   Quantitatively, the off-diagonal entries $\hat{\beta}_n$ of $J_N$ decay exponentially in $h = \ln(\tau)$, whereas the diagonal entries approach multiples of $h$. In other words, the Jacobi matrix equals a diagonal matrix plus an exponentially decaying correction, from which it follows via eigenvalue perturbation theory~\cite{kato2013perturbation} that the eigenvalues (roots) $x_k$ converge exponentially to the diagonal entries: $x_k \to kh$ as $\tau \to \infty$.   In fact, a more careful analysis shows that the convergence of $x_k$ to $kh$ is even faster than first-order perturbation theory might suggest, due to cancellations between first-order perturbation theory (which depends on the $kn + O(1/\tau)$ diagonal elements) and second-order perturbation theory (which depends on the square of the $O(1/\sqrt{\tau})$ off-diagonal elements). This rapid convergence is illustrated in \figref{MDLrootsdiff}, which shows $x_0$ empirically converging as $O(1/\tau^N)$ for $J_N$ with $N=10$.  A more precise asymptotic prediction of $x_k - kh$ may be an interesting topic for future work, but in practice the Gaussian summation scheme is most useful in the $\tau \approx 1$ ($hs \ll 1$) regime where naive summation is slowly converging.

In the small-$hs$ regime, where naive summation requires a huge number of terms $\sim 1/hs$, we will demonstrate below that MDL Gaussian summation retains rapid convergence.  As $hs \to 0^+$ ($\tau \to 1^+$), in fact, we recover the Laguerre polynomials~\cite{KytheSchaferkotter04,szego1939orthogonal} and the Gauss--Laguerre~\cite{sansone1959orthogonal} points and weights.  In this limit, we have
\begin{align*}
    h\hat{\alpha}_n &\approx \frac{2n+1}{s} + O(h^2) \\
    h\hat{\beta}_n &\approx \frac{n+1}{s} + O(h^2),
\end{align*}
which tends to the coefficients of the three-term recurrence for the continuous Laguerre polynomials~\cite{szego1939orthogonal} with a change of variables $x \to sx$ (to change the weight function from $e^{-x}$ to $e^{-sx}$).  The $O(h^2)$ convergence of the MDL recurrence to the Laguerre recurrence is unsurprising because the generalized MDL inner product \eqref{generalizedform} can be viewed as a trapezoidal rule approximation for the continuous Laguerre inner product $\int_0^\infty f(x) g(x)e^{-sx}dx$, and the trapezoidal rule converges as $O(h^2)$~\cite{atkinson2008introduction}.
\begin{figure}[t]
  \centerline{\includegraphics[width=0.65\columnwidth]{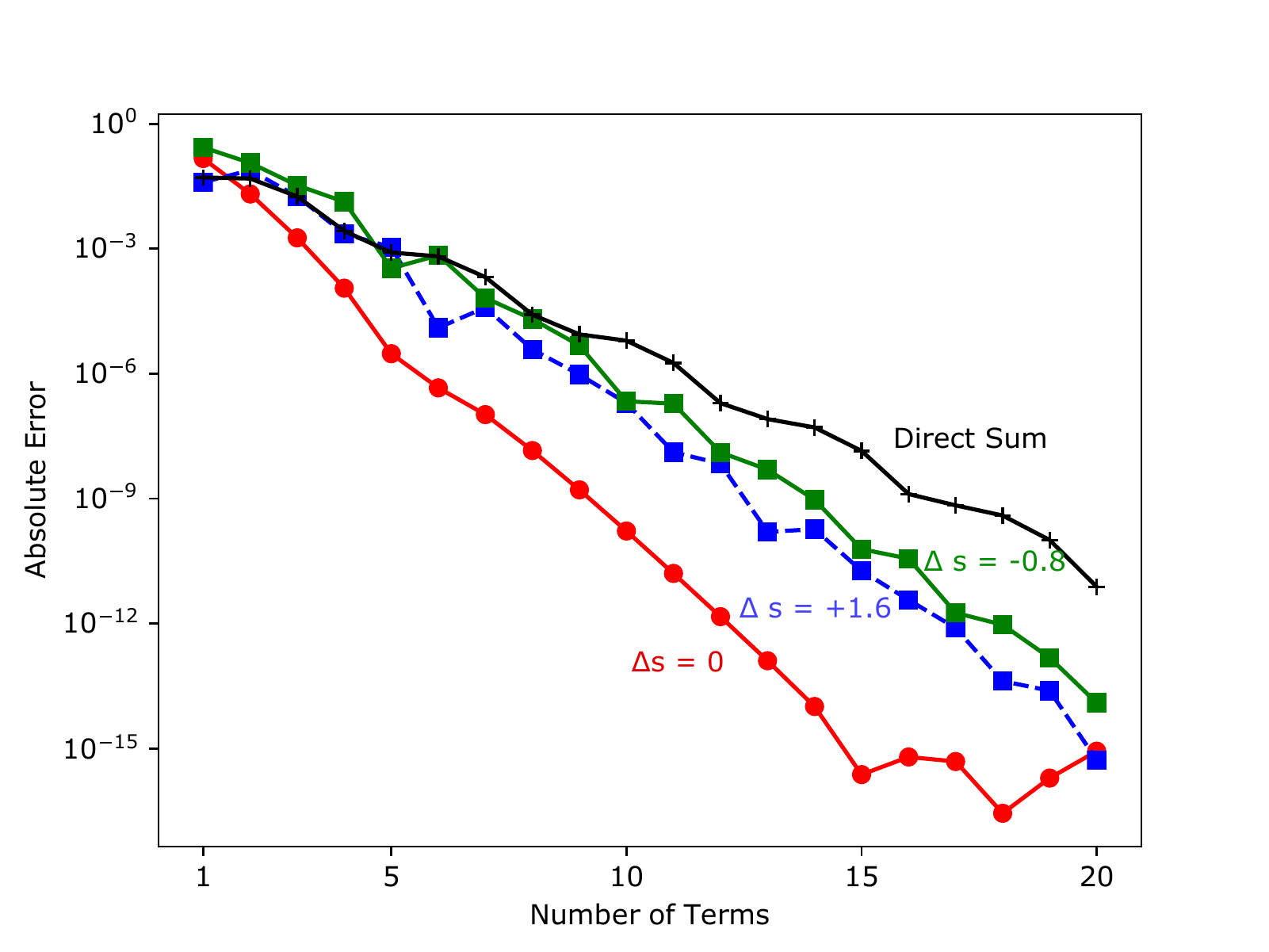}}
  \caption{Comparison of MDL quadrature convergence for the function $F(x) = \cos{(x)} e ^{-1.6x}$ via quadratures generated by $\tau = e^{-s} = e^{-(1.6+\Delta s)} \in \{ e^{-0.8}, e^{-1.6}, e^{-3.2}\}$ ($\Delta s \in \{-0.8, 0, 1.6\}$). While convergence remains exponential in all cases, optimal performance occurs at the correct asympototic decay rate $s = 1.6$ ($\Delta s = 0)$.}
  \label{fig:graph_qq1}
\end{figure}

\subsection{MDL quadrature convergence}

The convergence of Gaussian quadrature is well established for continuous measures on bounded intervals~\cite{Trefethen2008}
and on $[a, \infty)$~\cite{uspensky1928convergence,bultheel2000convergence}. Though~\citeasnoun{uspensky1928convergence} does not directly address discrete measures, in \cref{sec:convergenceproof} we adapt their arguments to demonstrate the convergence of our MDL quadrature for any polynomially bounded summand $F$, regardless of the decay rate $s$ used to derive the quadrature rule.

Nevertheless, in cases where $F(x)$ is asymptotically exponentially decaying proportional to $e^{-s_0 x}$, as in many physical applications of Matsubara summation, selecting the correct decay rate $s = s_0$ in the quadrature rule is useful to obtain the fastest possible convergence. We illustrate this property for a simple oscillating summand $F(x) = \cos{(x)}e^{-1.6x}$ ($s_0 = 1.6$) in Figure~\ref{fig:graph_qq1}, summed with $h=1$ ($\sum' F(n)$). Here, we constructed quadrature schemes using three different decay rates $s = 0.8$, $1.6$ ($= s_0$), and $3.2$. All three quadrature schemes converged for this summand, but the scheme assuming the correct decay rate $s = s_0 = 1.6$ (red) achieved superior convergence compared with either MDL summation with an ``incorrect'' $s$ (blue and green) or direct summation (black).

\section{Application to Casimir Forces}
\label{sec:apply_casimir}

Casimir forces, a generalization of van der Waals forces (between particles) to macroscopic objects, are an important interaction between electrically neutral surfaces at sub-$\mu m$ separations and arise as a consequence of quantum vacuum fluctuations in the electromagnetic fields~\cite{Lifshitz80,Rodriguez2011}.   Physically, the net force is an integral $\int_0^\infty d\omega$ over contributions from fluctuations at all frequencies~$\omega$ (or wavelengths $2\pi c/\omega$ where $c$ is the speed of light), weighted by a Bose--Einstein distribution $\coth(\hbar \omega / 2k_BT)$ where $\hbar$ is Planck's constant, $k_B$ is Boltzmann's constant, and $T$ is the temperature~\cite{Lifshitz80,Johnson11}.   Computationally, it is useful to ``Wick rotate'' this integral (which is wildly oscillatory on the real-$\omega$ axis) to the imaginary-frequency axis $\omega = i\xi$, in which case the integral becomes a sum \eqref{sum} over the ``Matsubara frequencies'' $\xi_n = \frac{2\pi k_B T}{\hbar} n = hn$ (where $h$ the spacing parameter of our notation in the previous sections, not Planck's constant).   This bosonic Matsubara sum also arises in various other contexts of statistical physics~\cite{Stefanucci13}.   In the limit of zero temperature ($T \to 0^+$), this sum becomes an integral $\int_0^\infty F(\xi) d\xi$ over the imaginary-frequency axis.

The summand $F(\xi)$ in the Casimir force represents the contribution of random current fluctuations at an frequency $\xi=-i\omega$ everywhere in the materials~\cite{Lifshitz80} and requires the solution of Maxwell's equations~\cite{Johnson11}.   This electromagnetic scattering problem has been formulated in a wide variety of ways in the Casimir-force literature depending upon the geometry in question~\cite{Johnson11}, but one of the most general formulations uses a boundary-element method (BEM)~\cite{reid2013fluctuating}, which discretizes each surface into a generic mesh (e.g. a triangular mesh) and formulates electromagnetism as a discretized scattering problem $Mu=f$, where the matrix $M$ is a Galerkin discretization of an integral operator based on the Maxwell Green's function~\cite{Harrington61}.   It turns out that the Casimir summand/integrand $F$ is given by a simple trace formula involving the matrix $M$~\cite{reid2013fluctuating}.   Therefore, the evaluation of $F$ requires computational effort equivalent to that of solving Maxwell's equations (in order to factorize the matrix $M$ in BEM, or to obtain a similar information such as a scattering matrix in other Casimir formulations), and it is desirable to compute the Matsubara sum with as few $F$ evaluations as possible, especially at low temperatures where the spacing between Matsubara frequencies is small.

The key to rapid evaluation of the Casimir Matsubara sum is to exploit the asymptotic decay rate of the summand $F(\xi)$, and this decay turns out to be very straightforward to obtain analytically.  At imaginary frequencies, Maxwell's equations yield \emph{exponentially decaying interactions}, because spherical waves $e^{-i\omega r/c}/r$ are transformed to decaying solutions $e^{-\xi r/c}/r$, with an exponential decay rate proportional to $\xi$~\cite{Rodriguez2011}. This causes the Casimir summand to decay exponentially.  More precisely, using a Born approximation, one can easily show that the large-$\xi$ interactions are dominated by a single ``round trip'' of waves from one object to the other and back~\cite{maghrebi2011diagrammatic}.  In consequence, for the force between two objects with minimum separation $d$, the Casimir summand decays at least as fast as $O(e^{-2\xi d / c})$ for large~$\xi \gg c/d$.   This asymptotic decay is illustrated for two spheres with $d=1\,\mu m$ in \figref{graphIntegrand}.

The $\xi=0$ term in the Casimir Matsubara sum requires special care, because a singularity arises in most formulations of Maxwell's equations at zero frequency, where the electric and magnetic fields decouple~\cite{Epstein2009}.   It turns out that the correct zero-frequency Casimir contribution should be computed as a limit $\xi \to 0^+$~\cite{Johnson11}.  Computationally, this limit can be obtained by methods such as Richardson extrapolation~\cite{Press07}.  A favorable property of MDL quadrature~\eqref{sum} for this application, however, is that the quadrature scheme \emph{implicitly} computes the $\xi \to 0^+$ limit, since it evaluates the summand only at strictly positive quadrature points $x_k > 0$, as illustrated in \figref{MDLrootsdiff}.

\begin{figure}[t]
  \centerline{\includegraphics[width=0.65\columnwidth]{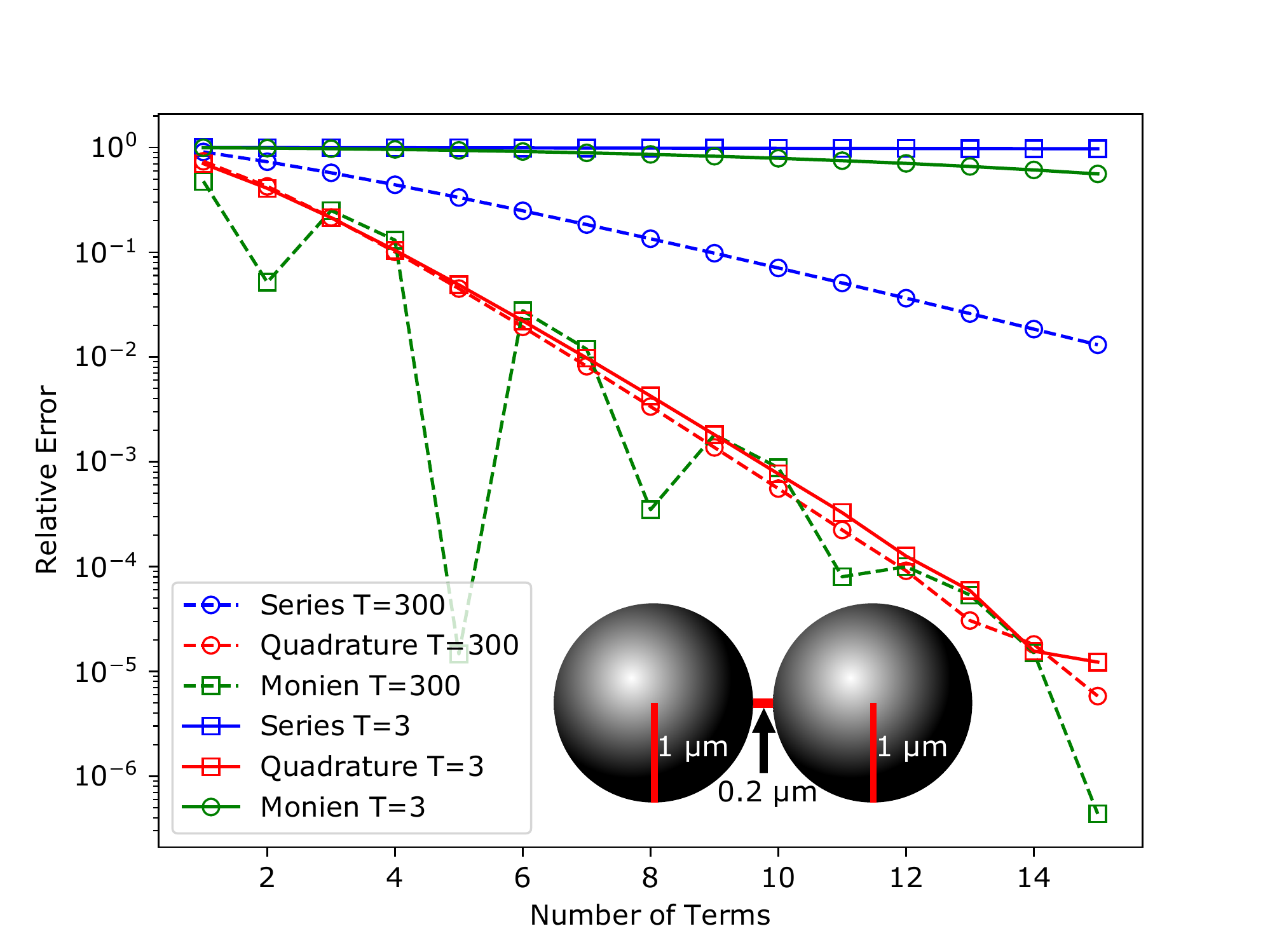}}
  \caption{Relative error of the series, quadrature, and Monien approaches at $T = 300, 3$ K. }
  \label{fig:graph4_mon}
\end{figure}

\subsection{Efficiency for Casimir summation}
\label{sec:casimir}

We compare three methods for approximating the Casimir Matsubara sum in \figref{graph4_mon}, to evaluate their accuracy as a function of the number $N$ of summand $F(\xi)$ evaluations.  The first approach (blue lines) is naive summation truncated to $N$ terms. The second method is our MDL quadature scheme of order $N$ (red lines).  The third method (green lines) is an alternative Gaussian-quadrature scheme proposed in~\cite{monien2010gaussian}, which uses a quadratically decaying weight function instead of our exponential weight function, so that it does not exploit knowledge of the asymptotic decay rate of~$F$.

Our test geometry consisted of two spheres of radius 1~$\mu$m separated by a gap of length 0.2~$\mu$m, a typical separation for nanoscale applications of Casimir forces. For both high ($T=300$~K) and low ($3$~K) temperatures, we computed the estimated Casimir force between the spheres using the free/open-source SCUFF-CAS3D BEM solver~\cite{reid2013fluctuating,SCUFF2}. In order to evaluate the summand at the smallest $\xi d/c \ll 1$ point (where our BEM solver encountered numerical difficulties for all three methods), we used Richardson extrapolation.

All three methods are viable in the high-temperature case $T = 300$~K, altbough naive summation is still noticeably less effective than the quadrature-based approaches. As expected, however, our method shows superior performance at low temperatures.  At $T=3$~K, the naive summation method requires hundreds of evaluations to obtain even one digit of accuracy, while the alternative quadrature method performs only marginally better. Our method, meanwhile, shows no increase in the number of required function evaluations.

\section{Conclusion}

Overall, our quadrature approach for computing this Matsubara sum is a significant improvement over naive summation and even compared to Gaussian quadrature schemes that do not exploit knowledge of the asymptotic decay rate. This approach could also be applied towards any exponentially decaying Matsubara sum (e.g. in density functional theory), or more generally to any similar summation problem.   More generally, given the MDL polynomials and an exponentially decaying function $F(x)$ evaluated at the MDL roots, one could construct a polynomial approximation that could be used to rapidly compute the sum or any other functional of $F(nh)$~\cite{TrefethenATAP}.

There are many possible avenues for future study.  A related problem is to evaluate the Matsubara sum at many different temperatures (different~$h$) for the same~$F$, and in this case one would want to construct a \emph{single} polynomial approximation and re-use that for multiple $h$ values---it would be interesting to consider what polynomial basis would be optimal for such a problem.  In our work, we used the Golub--Welsch algorithm to obtain Gaussian quadrature points and weights, but it might also be interesting to explore nested quadrature rules, such as Gauss--Kronrod or Gauss--Patterson rules~\cite{KytheSchaferkotter04}, in order to obtain error estimates and enable adaptive determination of the required quadrature order.  It would also be useful to develop tighter  theoretical estimates for the convergence rate of such quadrature schemes for polynomially bounded, analytic summands on infinite intervals, since previous such error bounds~\cite{Trefethen2008} typically apply only to quadrature on finite intervals.

\appendix
\section{Proofs of Results 2.2--2.3}
\label{sec:proofs}
Here, we provide extended proofs of the important results in the text pertaining to the MDL polynomials:


\begin{proof}[Proof of Lemma \ref{MDLscale}]
We proceed via strong induction. It is easy to see that when $n=0$, we have  $L_0'(0) = 1$ and $\langle L_0', L_0'\rangle ' = \frac{\tau+1}{2\tau-2} = \frac{1+\tau^1}{\tau^0(1+\tau^0)(\tau-1)}$.
Now, assume the statement for all $n' < n$. By Lemma~\ref{GS}, we have
\begin{equation*}
    L_n' = L_n + \frac{1}{2}\frac{1}{\tau^n}\sum\limits_{i=0}^{n-1} \frac{L_i'(0)L_i'}{\langle L_i', L_i'\rangle} = L_n + \frac{1}{\tau^n}\sum\limits_{i=0}^{n-1} \frac{\tau^i(\tau-1)L_i'}{(1+\tau^{i+1})},
\end{equation*}
which implies  \begin{align*}
    L_n'(0) &=  L_n(0) + \frac{1}{\tau^n}\sum\limits_{i=0}^{n-1} \frac{\tau^i(\tau-1)L_i'(0)}{(1+\tau^{i+1})} \\&= \frac{1}{\tau^n}\left( 1 + \sum\limits_{i=0}^{n-1} \frac{2\tau^i(\tau-1)}{(1+\tau^i)(1+\tau^{i+1})} \right) \\&= \frac{1}{\tau^n}\left( 1 + 2\sum\limits_{i=0}^{n-1}\left( \frac{1}{1+\tau^i}-\frac{1}{1+\tau^{i+1}}\right) \right) \\&= \frac{1}{\tau^n}\left( 1 + 2\left( \frac{1}{1+\tau^0}-\frac{1}{1+\tau^n}\right) \right) \\&= \frac{1}{\tau^n}\left(\frac{2\tau^n}{1+\tau^n}\right) = \frac{2}{1+\tau^n}.
\end{align*}
Similarly, by the orthogonality of the $L_i'$, we deduce that \begin{align*}
    \langle L_n', L_n' \rangle ' &= \langle L_n', L_n  \rangle ' \\&= \langle L_n', L_n  \rangle - \frac{1}{2}L_n'(0)L_n(0) \\&= \langle L_n, L_n  \rangle - \frac{1}{\tau^n(1+\tau^n)} \\&= \frac{1}{\tau^{n-1}(\tau-1)} - \frac{1}{\tau^n(1+\tau^n)}  = \frac{1+\tau^{n+1}}{\tau^n(1+\tau^{n})(\tau-1)},
\end{align*}
\noindent as desired.
\end{proof}

\begin{proof}[Proof of Theorem \ref{recurrence}]
Note that the leading coefficient of $L_n'$ is $\left(\frac{1}{\tau} - 1\right)^n\frac{1}{n!}$. This shows that $ \alpha_n = \frac{\tau-1}{(n+1)\tau}$. Furthermore, since $\langle L_{n+1}', L_{n-1}' \rangle' = 0$, we have  \begin{align*} \gamma_n\langle L_{n-1}', L_{n-1}' \rangle' &= \alpha_n\langle xL_{n}', L_{n-1}' \rangle' \\&= \alpha_n\langle L_{n}', xL_{n-1}' \rangle' \\&= \frac{\alpha_n}{\alpha_{n-1}}\langle L_n', L_n' \rangle',
\end{align*} implying that \begin{align*}
    \gamma_n &= \frac{\alpha_n}{\alpha_{n-1}}\frac{\langle L_{n}', L_{n}' \rangle'}{\langle L_{n-1}', L_{n-1}' \rangle'} \\&= \frac{n}{n+1}\frac{(1+\tau^{n+1})(1+\tau^{n-1})}{\tau(1+\tau^n)^2}.
\end{align*}
We then have \begin{align*}
    \frac{2}{1+\tau^{n+1}} &= L_{n+1}'(0) \\&= \beta_n L_n'(0) - \gamma_n L_{n-1}'(0) \\&= \beta_n\frac{2}{1+\tau^n} - \frac{2}{1+\tau^{n-1}}\frac{n}{n+1}\frac{(1+\tau^{n+1})(1+\tau^{n-1})}{\tau(1+\tau^n)^2} \\&= \beta_n\frac{2}{1+\tau^n} - \frac{n}{n+1}\frac{2(1+\tau^{n+1})}{\tau(1+\tau^n)^2},
\end{align*} which implies \begin{align*}
    \beta_n = (1+\tau^n)\left(\frac{1}{1+\tau^{n+1}} + \frac{n}{n+1}\frac{(1+\tau^{n+1})}{\tau(1+\tau^n)^2}\right) = \frac{1+\tau^n}{1+\tau^{n+1}} + \frac{n}{\tau(n+1)}\frac{1+\tau^{n+1}}{1+\tau^n}.
\end{align*}
\end{proof}

\section{Convergence of MDL Quadrature}
\label{sec:convergenceproof}
Here, we establish the convergence of our MDL quadrature (summation) scheme. For simplicity, we consider the case of $h=s=1$, and prove convergence for any summand $F(x)$ satisfying the asymptotic upper bound
$$
|F(x)| = e^{-x}|f(x)| = O\left(\frac{1}{x^{1+\epsilon}}\right)
$$
for some $\epsilon > 0$. Our proof is an extension of the proof in~\citeasnoun{uspensky1928convergence}. Below, we will let $W(x)$ denote the function taking the value $\frac{1}{2}$ at $x = 0$, $e^{-x}$ whenever $x \in \mathbb{Z}^+$, and $0$ elsewhere, and let $w(x)$ denote the corresponding distribution $w(x)=\sum_{n=0}^\infty W(n)\delta(x-n)=\sum_{n=0}^\infty{}' \delta(x-n)e^{-n}$.

\citeasnoun{uspensky1928convergence} establishes convergence based on properties of the even moments $c_{2m} = \int x^{2m} w(x) dx$ of the weight function, and we can employ the same strategy here.
Since the function $x^ne^{-x}$ is increasing in $(0, n)$ and decreasing in $(n, \infty)$, we can bound the even moments $c_{2m}$ for $m>0$ by breaking them into left and right Riemann sums as follows:

\begin{align*}c_{2m} = \sum\limits_{k=1}^\infty k^{2m}e^{-k} &= \sum\limits_{k=0}^{2m-1} k^{2m}e^{-k} + \sum\limits_{k=2m+1}^\infty k^{2m}e^{-k} + \left(\frac{2m}{e}\right)^{2m} \\&\leq \left(\frac{2m}{e}\right)^{2m} + \int\limits_{0}^{\infty} x^{2m}e^{-x}dx  \\&= \left(\frac{2m}{e}\right)^{2m} + (2m)! < 2(2m!)1^{2m}.\end{align*}
This bound coincides with the condition given by~\citeasnoun{uspensky1928convergence} (p.~556) with $C = 2$ and $R = 1$. We then apply identical reasoning to~\citeasnoun{uspensky1928convergence} (p.~546) to show that, for any non-integer~$H$ and quadrature points and weights $(x_k, w_k)$ generated by orthogonal polynomials of increasing order $n$, $$ \lim\limits_{n \rightarrow \infty} \sum\limits_{x_k  \geq H} w_k = \int\limits_H^\infty w(x)dx = \sum_{k>H} W(k) \, .$$ For any nonnegative integer $m$ and $0 < \epsilon_m < 1$, we therefore see that  $$ \lim\limits_{n \rightarrow \infty} \sum\limits_{ x_k \in (m - \epsilon_m, m + \epsilon_m)} w_k = \int\limits_{m - \epsilon_m}^{m+\epsilon_m} w(x)dx = W(m) > 0 \, , $$ and hence there must be at least one $x_k$ in any neighborhood of $m$ for sufficiently large quadrature order $n$, with combined weight approaching $W(m)$. When combined with the facts that each quadrature point $x_k$ must be positive~\cite{engblom2006gaussian} and each interval $(m, m+1)$ between nonnegative integers must contain at most one quadrature point~\cite{gottlieb1938concerning, popoviciu1929distribution}, we see that the MDL points and weights necessarily converge to the nonnegative integers and $W(m)$, respectively.

Finally, we may conclude using the same methods as~\citeasnoun{uspensky1928convergence} that the MDL quadrature converges over the infinite interval: For any $\hat{\epsilon} > 0$, we may choose some non-integer $H$ such that $$\int\limits_H^\infty w(x)f(x) dx < \hat{\epsilon} \, .$$ We then see that $$\lim\limits_{n \rightarrow \infty} \sum\limits_{x_k < H} w_kf(x_k) = \int\limits_0^H w(x)f(x)dx$$ by our previous result, and the convergence of the remaining integral $\int\limits_H^\infty w(x)f(x) dx $ may be handled using the same technique as~\citeasnoun{uspensky1928convergence} (p.~557). This method makes use of the bounds on the moments $c_{2m}$ proven previously.

\section*{Acknowledgments}
This work was supported in part by the Undergraduate Research Opportunities program at MIT.  We are also grateful to Alan~Edelman for helpful discussions.

\bibliographystyle{siamplain}
\bibliography{biblio}
\end{document}